\documentclass{amsart}

\usepackage{amsmath}
\usepackage{amssymb}
\usepackage{bibentry}
\usepackage{enumitem}
\usepackage{graphicx}
\usepackage{mathrsfs}
\usepackage{tabu}
\usepackage{tikz}
\usepackage[comma,numbers,sort,square]{natbib}
\usepackage[colorlinks=true,
            linktocpage=true,
            linkcolor=magenta,
            citecolor=magenta,
            urlcolor=magenta]{hyperref}
\PassOptionsToPackage{hyphens}{url}

\usetikzlibrary{arrows}
\usetikzlibrary{shapes}
\usetikzlibrary{snakes}

\newtheorem{theorem}{Theorem}[section]
\newtheorem{proposition}[theorem]{Proposition}
\newtheorem{lemma}[theorem]{Lemma}
\newtheorem{corollary}[theorem]{Corollary}
\theoremstyle{definition}
\newtheorem{definition}[theorem]{Definition}
\newtheorem{convention}[theorem]{Convention}

\newtheorem{remark}[theorem]{Remark}

\DeclareMathOperator{\res}{\upharpoonright}

\newcommand{\dom}{\operatorname{dom}}

\newcommand{\seq}[1]{\langle #1 \rangle}

\newcommand{\cred}{\leq_{\text{\upshape c}}}

\newcommand{\scred}{\leq_{\text{\upshape sc}}}

\newcommand{\uequiv}{\equiv_{\text{\upshape W}}}
\newcommand{\ured}{\leq_{\text{\upshape W}}}

\newcommand{\suequiv}{\equiv_{\text{\upshape sW}}}
\newcommand{\sured}{\leq_{\text{\upshape sW}}}
\newcommand{\nsured}{\nleq_{\text{\upshape sW}}}

\newcommand{\sjoin}{\boxplus}
\newcommand{\eval}{\operatorname{e}}
\newcommand{\diff}{\infty}
\newcommand{\monapprox}{\operatorname{a}}
\newcommand{\infcomp}[1]{\widetilde{#1}}
\newcommand{\Tred}{\leq_{\mathrm{T}}}
\newcommand{\nTred}{\nleq_{\mathrm{T}}}

\begin{document}

\title{Joins in the strong Weihrauch degrees}

\author{Damir D. Dzhafarov}
\address{Department of Mathematics\\
University of Connecticut\\
%Storrs, Connecticut U.S.A.}
341 Mansfield Road\\ Storrs, Connecticut 06269-1009 U.S.A.}
%\curraddr{}
\email{damir@math.uconn.edu}

\thanks{}

\begin{abstract}
The Weihrauch degrees and strong Weihrauch degrees are partially ordered structures representing degrees of unsolvability of various mathematical problems. Their study has been widely applied in computable analysis, complexity theory, and more recently, also in computable combinatorics. We answer an open question about the algebraic structure of the strong Weihrauch degrees, by exhibiting a join operation that turns these degrees into a lattice. Previously, the strong Weihrauch degrees were only known to form a lower semi-lattice. We then show that unlike the Weihrauch degrees, which are known to form a distributive lattice, the lattice of strong Weihrauch degrees is not distributive. Therefore, the two structures are not isomorphic.
\end{abstract}

\thanks{The author was supported in part by NSF Grant DMS-1400267. He thanks Vasco Brattka for numerous helpful comments and suggestions during the preparation of this paper, and in particular, for noticing an error in an earlier version of Proposition \ref{prop:eval}.}

\maketitle

\section{Introduction}

Weihrauch reducibility provides a framework for measuring the relative complexity of solving certain mathematical problems, and in particular, of telling when the task of solving one mathematical problem can be reduced to the task of solving another. The program of classifying mathematical problems using Weihrauch reducibility was initiated by Brattka and Gherarrdi~\cite{BG-2011b} and Gherardi and Marcone~\cite{GM-2008}. Weihrauch reducibility itself goes back to Weihrauch~\cite{Weihrauch-1992}, and has been widely deployed in computable analysis. More recently, the concept was independently re-discovered by Dorais, Dzhafarov, Hirst, Mileti, and Shafer~\cite{DDHMS-2016} in the context of computable combinatorics. The classification program can be seen as a foundational one, in the spirit of Friedman and Simpson's program of reverse mathematics (cf.\ Simpson~\cite{Simpson-2009}). In many ways, Weihrauch reducibility leads to a refinement and extension of reverse mathematics; see Hirschfeldt~\cite[Section 2.2]{Hirschfeldt-2014} or Hirschfeldt and Jockusch~\cite[Section 1]{HJ-2016} for detailed discussions.

Intuitively, a mathematical problem $\mathsf{P}$ consists of a collection of \emph{instances}, and for each instance, a collection of \emph{solutions} to this instance (in that problem).
%We can take the instances and solutions to be, e.g., subsets of the natural numbers, though the formal definition (given in the next section) is more general.
Given math problems $\mathsf{P}_0$ and $\mathsf{P}_1$, we can then informally define $\mathsf{P}_0$ to be \emph{strongly Weihrauch reducible} to $\mathsf{P}_1$ if there is an effective way to convert every instance $p$ of $\mathsf{P}_0$ into an instance $\widetilde{p}$ of $\mathsf{P}_1$, and an effective way to convert every solution $\widetilde{q}$ to $\widetilde{p}$ in $\mathsf{P}_1$ into a solution $q$ to $p$ in $\mathsf{P}_0$. This method of reducing the task of solving $\mathsf{P}_0$ to that of solving $\mathsf{P}_1$ is natural, and shows up frequently throughout mathematics (see, e.g., \cite{DDHMS-2016}, Section 1, for some specific examples). It is, however, somewhat restrictive in that the backward conversion is not allowed access to the original instance $p$ of $\mathsf{P}_0$. For this reason, we also define $\mathsf{P}_0$ to be \emph{Weihrauch reducible} to $\mathsf{P}_1$ if there is an effective way to convert every instance $p$ of $\mathsf{P}_0$ into an instance $\widetilde{p}$ of $\mathsf{P}_1$, and an effective way to convert $p$, together with any solution $\widetilde{q}$ to $\widetilde{p}$ in $\mathsf{P}_1$, into a solution $q$ to $p$ in $\mathsf{P}_0$. Both types of reductions have been examined at length in the literature, with the past few years in particular seeing a surge of interest. An updated bibliography of publications contributing to this study is maintained by Brattka~\cite{Brattka-bib}. (See also Dzhafarov~\cite{Dzhafarov-2015, Dzhafarov-2016}, and Remark \ref{rem:nonunif} below, for a non-uniform version of Weihrauch reducibility; and see Pauly~\cite{Pauly-2010} for a version in which computable transformations are replaced by continuous ones.)

In this paper, we focus on the algebraic structure of these reducibilities. For Weihrauch reducibility, this has been studied extensively, e.g., by Brattka and Gherardi~\cite{BG-2011}, Brattka and Pauly~\cite{BP-TA}, Higuchi and Pauly~\cite{HP-2013}, and others. We focus here on strong Weihrauch reducibility. It is known that the Weihrauch \emph{degrees} (i.e., the equivalence classes under Weihrauch reducibility) form a lattice under certain natural operations (see Theorem~\ref{thm:Wlatt} below). We prove the corresponding result for the strong Weihrauch degrees, thereby answering an open question (see, e.g., Brattka~\cite{Brattka-2015}, or H\"{o}lzl and Shafer~\cite{HS-2015}, Section 2). Further, we show that as in the case of the Weihrauch lattice, every countable distributive lattice can be embedded into the strong Weihrauch lattice. However, unlike in the Weihrauch case, we show that the strong Weihrauch lattice is itself not distributive. Hence, in particular, the Weihrauch degrees and strong Weihrauch degrees are not isomorphic structures.

The paper is organized as follows. In Section~\ref{S:defns}, we give some general background about Weihrauch reducibility, including precise definitions of the Weihrauch and strong Weihrauch degrees. In Section~\ref{S:main}, we define the supremum (join) operation on the strong Weihrauch degrees, and prove our main result that this turns the strong Weihrauch degrees into a lattice. Finally, in Section~\ref{S:distr}, we prove the non-distributivity of this lattice, and consider lattice embeddings.

\section{Background}\label{S:defns}

Our notation and terminology is mostly standard, following, e.g.,\ Soare~\cite{Soare-2016} and Weihrauch~\cite{Weihrauch-1987}. Throughout, we identify subsets of $\omega$ with their characteristic functions, and so regard them as elements of $2^\omega$. For convenience, if $p \in 2^\omega$ and $n \in \omega$, we will frequently write $n \in p$ and $n \notin p$ instead of $p(n) = 1$ and $p(n) = 0$, respectively, and refer to $n$ as being or not being an element of $p$. We let $\seq{\cdot,\cdot}$ denote the standard computable pairing function on $\omega$, and also the effective join on $2^\omega$ (in place of the more commonly used symbol $\oplus$). For $p \in 2^\omega$, we write $\seq{0,p}$ and $\seq{1,p}$ for $\seq{\{0\},p}$ and $\seq{\{1\},p}$, respectively. For a finite binary string $\sigma$ and a bit $i < 2$, we write $\sigma i^\omega$ for the element $p \in 2^\omega$ with $p(n) = \sigma(n)$ for all $n < |\sigma|$ and $p(n) = i$ for all $n \geq |\sigma|$. In particular, we write $0^\omega$ and $1^\omega$ for all the all-$0$ and all-$1$ infinite binary sequence, respectively.

For Turing functionals, we follow the following conventions.

\begin{convention}\label{con:functionals}
	Let $\Phi$ be a Turing functional and $p \in 2^\omega$.
	\begin{itemize}
		\item If $\Phi^p(n)[s] \downarrow$ for some $n,s \in \omega$ then also $\Phi^p(m)[s] \downarrow$ for all $m < n$.
		\item For each $s \in \omega$ there is at most one $n \in \omega$ for which $s$ is least such that $\Phi^p(n)[s] \downarrow$.
	\end{itemize}
\end{convention}

\noindent Further, we regard all Turing functionals as being $\{0,1\}$-valued, so that if $\Phi^p$ is total for some $\Phi$ and $p$, then $\Phi^p$ is an element of $2^\omega$.

We use $2^\omega$ here merely as a convenience, but will not rely on any of its specific properties as a topological space. Hence, everything in our treatment would go through equally well for Baire space in place of Cantor space.

We shall consider functions below which can take on multiple values, called \emph{multifunctions}. Formally, a multifunction $f$ from a set $X$ to a set $Y$, denoted $f : X \rightrightarrows Y$, represents that the value of $f(x)$ for each $x \in X$ is a subset of $Y$. If $f$ is a function or multifunction with domain a (possibly proper) subset of $X$, we denote this by $f : \subseteq X \to Y$ or $f : \subseteq X \rightrightarrows Y$, respectively. In this case, we refer to $f$ as a \emph{partial} function/multifunction (on $X$), and we denote its domain by $\dom(f)$.

A multifunction $f : X \rightrightarrows Y$ thus formalizes the concept of a mathematical problem, as was informally discussed in the introduction. The elements of $X$ are regarded as the instances of this problem, and for each $x \in X$, the elements of $f(x)$ are regarded as the solutions to the instance $x$ (in the problem $f$).

Unlike in the introduction, there are no restrictions above that the domains and co-domains of problems be subsets of the natural numbers, which was necessary in order to define computations from instances and solutions. The following definition will allow us to develop computability theory on a broader class of spaces for when we define Weihrauch reducibility below.

\begin{definition}
	A \emph{representation} of a set $X$ is a partial surjective function $\delta : \subseteq 2^{\omega} \to X$. The pair $(X,\delta)$ is a \emph{represented space}.
\end{definition}

Given represented spaces $(X_0,\delta_{X_0})$ and $(X_1,\delta_{X_1})$, we can define $\delta_{X_0 \sqcup X_1} : \subseteq 2^\omega \to X_0 \sqcup X_1$ by $\delta_{X_0 \sqcup X_1}(\seq{i,p}) = \seq{i,\delta_{X_i}(p)}$ for each $i < 2$ and $p \in 2^\omega$, and $\delta_{X_0 \times X_1} : \subseteq 2^\omega \to X_0 \times X_1$ by $\delta_{X_0 \times X_1}(\seq{p_0,p_1}) = \seq{\delta_{X_0}(p_0), \delta_{X_1}(p_1)}$ for all $p_0,p_1 \in 2^\omega$. These provided representations for $X_0 \sqcup X_1$ and $X_0 \times X_1$, respectively, which will play an important role in our work below.

\begin{definition}
	Let $f : \subseteq (X,\delta_X) \rightrightarrows (Y,\delta_Y)$ be a partial multifunction on represented spaces. A function $F : \subseteq 2^{\omega} \to 2^{\omega}$ is \emph{realizer} of $f$, in symbols $F \vdash f$, if $\delta_Y F(p) \in f \delta_X(p)$ for all $p \in \dom(f \delta_X)$.
\end{definition}

\begin{definition}\label{def:weihr}
	Let $f$ and $g$ be partial multifunctions on represented spaces.
	\begin{itemize}
		\item $f$ is \emph{Weihrauch reducible} to $g$, in symbols $f \ured g$, if there are Turing functionals $\Phi,\Psi : \subseteq 2^{\omega} \to 2^{\omega}$ such that $\Psi \seq{\mathrm{id}, G \Phi} \vdash f$ for all $G \vdash g$.
		\item $f$ is \emph{strongly Weihrauch reducible} to $g$, in symbols $f \sured g$, if there are Turing functionals $\Phi,\Psi : \subseteq 2^{\omega} \to 2^{\omega}$ such that $\Psi G \Phi \vdash f$ for all $G \vdash g$.	
	\end{itemize}
	If the above applies, we say that $f \ured g$ or $f \sured$ \emph{via} $\Phi$ and $\Psi$.
\end{definition}

The notion above formalizes the intuitive one of reducing one mathematical problem to another, as discussed in the introduction. We give an alternative definition, due to Dorais et al.~\cite[Definition 1.5]{DDHMS-2016}, in Section~\ref{S:distr}.

It is customary to refer to equivalence classes of under $\ured$ and $\sured$ as the \emph{Weihrauch degrees} and \emph{strong Weihrauch degrees}, respectively. Formally, of course, these objects are not sets but proper classes. Thus, we implicitly identify each partial multifunction $f : \subseteq (X,\delta_X) \rightrightarrows (Y,\delta_Y)$ with $\delta_Y^{-1} \circ f \circ \delta_X : \subseteq 2^\omega \to 2^\omega$, whereby the (strong) Weihrauch degrees can be regarded just as equivalence classes of multifunctions on Cantor space.

The following definition gives two important operations on multifunctions.

\begin{definition}
	Let $f : \subseteq (X_0,\delta_{X_0}) \rightrightarrows (Y_0,\delta_{Y_0})$ and $g : \subseteq (X_1,\delta_{X_1}) \rightrightarrows (Y_1,\delta_{Y_1})$ be partial multifunctions on represented spaces.
	\begin{itemize}
		\item $f \sqcup g : \subseteq (X_0 \sqcup X_1,\delta_{X_0 \sqcup X_1}) \rightrightarrows (Y_0 \sqcup Y_1,\delta_{Y_0 \sqcup Y_1})$ is defined by
		\[
			f(\seq{0,x}) = \{0\} \times f(x)
		\]
		for all $x \in \dom(f)$, and
		\[
			f(\seq{1,x}) = \{1\} \times g(x)
		\]
		for all $x \in \dom(g)$.
		\item $f \sqcap g : \subseteq (X_0 \times X_1,\delta_{X_0 \times X_1}) \rightrightarrows (Y_0 \sqcup Y_1,\delta_{Y_0 \sqcup Y_1})$ is defined by
		\[
			f(\seq{x,y}) = (\{0\} \times f(x)) \cup (\{1\} \times g(y))
		\]
		for all $x \in \dom(f)$ and $y \in \dom(g)$.
	\end{itemize}
\end{definition}

To save on notation, given a degree structure defined as the set of equivalence classes under some reducibility, we identify degree-invariant operations on the elements of the underlying space with operations on the degrees themselves. It is easy to see that both $\sqcup$ and $\sqcap$ are invariant under $\uequiv$, and we have the following result establishing their main properties.

\begin{theorem}[Pauly~\cite{Pauly-2010}, Theorem 4.22; Brattka and Gherardi~\cite{BG-2011b}, Theorem 3.14]\label{thm:Wlatt}
	The Weihrauch degrees form a bounded distributive lattice under $\ured$, with $\sqcup$ as supremum and $\sqcap$ as infimum.
\end{theorem}

The proof of the theorem also shows that $\sqcap$ gives the infimum operation for the strong Weihrauch degrees, and hence that these form a lower semi-lattice. The precise definitions of the top and bottom elements in Weihrauch and strong Weihrauch degrees are somewhat complicated, but as these are not needed for our work here, we refer the reader to~\cite[Sections 2.1]{BP-TA} for details.

\section{Main construction}\label{S:main}

We begin in this section with a series of computability-theoretic definitions, leading up to the definition of the supremum operation in the strong Weihrauch degrees.

\begin{definition}\label{def:monapprox}
	A \emph{monotone approximation}\footnote{Monotone approximations were also considered, in an unrelated context, by Dzhafarov and Igusa~\cite{DI-2017}, where they were called \emph{partial oracles}.} is an element $a \in 2^\omega$ with the following properties:
	\begin{itemize}
		\item every element of $a$ is of the form $\seq{n,s,i}$, where $n,s \in \omega$ and $i < 2$;
		\item for all $n,s,t \in \omega$ and $i,j < 2$, if $\seq{n,s,i}, \seq{n,t,j} \in a$ then $s = t$ and $i = j$;
		\item for all $m,n,s,t \in \omega$ and $i,j < 2$, if $\seq{m,s,i}, \seq{n,t,j} \in a$ and $m < n$ then $s < t$.
	\end{itemize}
	The monotone approximation is \emph{total} if for every $n \in \omega$ there is an $s \in \omega$ and $i < 2$ such that $\seq{n,s,i} \in a$.
\end{definition}

%\noindent 

An important class of monotone approximations for our purposes come from Turing computations.

\begin{definition}
	Given a Turing functional $\Psi$ and $p \in 2^\omega$, let $\monapprox_{\Psi(p)} \in 2^\omega$ consist of all the $\seq{n,s,i}$ where $n,s \in \omega$, $i < 2$, and $s$ is least such that $\Psi(p)(n)[s] \downarrow = i$.	
\end{definition}

Convention~\ref{con:functionals} ensures that $\monapprox_{\Psi(p)}$ is indeed a monotone approximation, as well as the following basic facts.

\begin{proposition}\label{prop:approx}
	For each Turing functional $\Psi$ and each $p \in 2^\omega$, $\monapprox_{\Psi(p)}$ is uniformly computable from $p$ and (an index for) $\Psi$. Further, if $\Psi(p)$ is total (i.e., is an element of $2^\omega$) then $\monapprox_{\Psi(p)}$ is total as a monotone approximation.
\end{proposition}

\begin{proof}
	Immediate.	
\end{proof}

\begin{definition}
	Let $\eval : \subseteq 2^\omega \to 2^\omega$ be the partial function with domain the set of all total monotone approximations $a$, such that for any such $a$ and all $n \in \omega$ and $i < 2$ we have $\eval(a)(n) = i$ if and only if $\seq{n,s,i} \in a$ for some $s \in \omega$.
\end{definition}

\begin{proposition}\label{prop:eval}
\
	\begin{enumerate}
		\item The partial function $\eval$ is a Turing functional.
		\item For each Turing functional $\Psi$ and each $p \in 2^\omega$, if $\Psi(p)$ is total then $\eval(\monapprox_{\Psi(p)}) = \Psi(p)$.
	\end{enumerate}
\end{proposition}

\begin{proof}
	For part (1), fix $p \in 2^\omega$. We can uniformly computably check, for each $k \in \omega$, whether the defining conditions of $p$ being a monotone approximation hold for all numbers less than $k$. Now for each $k$ for which this is the case and for each $n < k$, $\eval(p)(n)$ is computed by searching through $p$ until, if ever, an $s \in \omega$ and $i < 2$ are found with $\seq{n,s,i} \in p$, in which case the output is $i$. Thus, if $p$ is a total monotone approximation, $\eval(p)(n)$ will be defined for all $n$, and hence $\eval(p)$ will be defined as an element of $2^\omega$. Otherwise, either $p$ fails to be a monotone approximation, or it fails to be total, and in both cases $\eval(p)(n)$ will be undefined for some $n$. Thus, $\eval$ is a Turing functional with the desired domain.
	
	For part (2), note that by Proposition~\ref{prop:approx}, $\monapprox_{\Psi(p)}$ is total, so $\eval(\monapprox_{\Psi(p)})$ is an element of $2^\omega$. Now by definition, for all $n \in \omega$ we have that $\eval(\monapprox_{\Psi(p)})(n) = i$ if and only if $\seq{n,s,i} \in \monapprox_{\Psi(p)}$ for some $s$, if and only if $\Psi(p)(n) \downarrow = i$.
\end{proof}

In what follows, if $Y$ is any set, we use $\diff_Y$ to denote a fixed element not in $Y$.

\begin{definition}
	Let $(Y,\delta_Y)$ be a represented space.
	\begin{enumerate}
		\item Let $\infcomp{Y} = Y \cup \{\diff_Y\}$.
		\item Define $\delta_{\infcomp{Y}} : 2^\omega \to \infcomp{Y}$ by
			\[
				\delta_{\infcomp{Y}}(p) =
				\begin{cases}
					\delta_{Y} \eval(p) & \text{if } p \in \dom(\eval) \text{ and } \eval(p) \downarrow \in \dom(\delta_Y),\\
					\diff_Y & \text{otherwise}
				\end{cases}
			\]
			for all $p \in 2^\omega$.
	\end{enumerate}

\end{definition}

Clearly, $\delta_{\infcomp{Y}}$ is a representation of $\infcomp{Y}$. The definition gives rise to the following operation on partial multifunctions.

\begin{definition}
	Let $f : \subseteq (X_0,\delta_{X_0}) \rightrightarrows (Y_0,\delta_{Y_0})$ and $g : \subseteq (X_1,\delta_{X_1}) \rightrightarrows (Y_1,\delta_{Y_1})$ be partial multifunctions on represented spaces. We define
	\[
		f \sjoin g : \subseteq (X_0 \sqcup X_1,\delta_{X_0 \sqcup X_1}) \rightrightarrows (\infcomp{Y_0} \times \infcomp{Y_1},\delta_{\infcomp{Y_0} \times \infcomp{Y_1}})
	\]
	by
	\[
		(f \sjoin g)(\seq{0,x}) = f(x) \times \infcomp{Y_1}
	\]
	for all $x \in \dom(X_0)$, and
	\[
		(f \sjoin g)(\seq{1,x}) = \infcomp{Y_0} \times g(x)
	\]
	for all $x \in \dom(X_1)$. 
\end{definition}

It is not difficult to check that $\sjoin$ is invariant, commutative, and associative, up to strong Weihrauch equivalence. We are now ready to prove our main theorem, that the above definition gives the supremum operation on the strong Weihrauch degrees\footnote{The definition of $\sjoin$ has a product on the input side and a co-product on the output, unlike the definition of $\sqcup$, which has a product on both sides. (In this sense, $\sjoin$ is dual to the definition of $\sqcap$, which has a co-product on the input and a product on the output.) We learn from Vasco Brattka [personal communication] that Peter Hertling and he also attempted to construct a supremum for the strong Weihrauch degrees of this form. Unfortunately, their approach was unsuccessful because it did not consider the move from $Y_0 \times Y_1$ to the completed space $\infcomp{Y}_0 \times \infcomp{Y}_1$. This points to the importance of the addition of the $\diff_{Y_0}$ and $\diff_{Y_1}$ elements.}.

\begin{theorem}\label{thm:main}
	Let $f : \subseteq (X_0,\delta_{X_0}) \rightrightarrows (Y_0,\delta_{Y_0})$ and $g : \subseteq (X_1,\delta_{X_1}) \rightrightarrows (Y_1,\delta_{Y_1})$ be partial multifunctions on represented spaces. Then $f \sjoin g$ is the supremum of $f$ and $g$ under $\sured$.
\end{theorem}

\begin{proof}
	Fix $f$ and $g$. We divide our proof into the following two lemmas.
	
	\begin{lemma}\label{lem:main1}
		$f \sured f \sjoin g$ and $g \sured f \sjoin g$.
	\end{lemma}
	
	\begin{proof}
		For each $i < 2$, let $\Phi_i : 2^\omega \to 2^\omega$ be the map $p \mapsto \seq{i,p}$, and let $\Psi_i : 2^\omega \to 2^\omega$ be the map $\seq{q_0,q_1} \mapsto \eval(q_i)$. Note that $\Psi_0$ and $\Psi_1$ are Turing functionals by Proposition~\ref{prop:eval}. We show that $f \sured f \sjoin g$ via $\Phi_0$ and $\Psi_0$; a symmetric argument shows that $g \sured f \sjoin g$ via $\Phi_1$ and $\Psi_1$.
	
		Suppose $H \vdash f \sjoin g$ and fix any $p \in \dom(f \delta_{X_0})$. We must show that
		\[
			\delta_{Y_0} \Psi_0 H \Phi_0(p) = \delta_{Y_0} \Psi_0 H(\seq{0,p}) \in f \delta_{X_0}(p).
		\]
		Since $H \vdash f \sjoin g$ and $\seq{0,p} \in \dom((f \sjoin g)(\delta_{X_0 \sqcup X_1}))$, we have
		\begin{eqnarray*}
			\delta_{\infcomp{Y_0} \times \infcomp{Y_1}}  H(\seq{0,p}) & \in & (f \sjoin g)(\delta_{X_0 \sqcup X_1})(\seq{0,p})\\
			& = & (f \sjoin g)(\seq{0,\delta_{X_0}(p)})\\
			& = & f \delta_{X_0}(p) \times Y_1,
		\end{eqnarray*}
		Thus, $\delta_{\infcomp{Y_0} \times \infcomp{Y_1}}  H(\seq{0,p})$ is a pair $\seq{y,z}$ with $y \in Y_0$, so in particular, $y \neq \diff_{Y_0}$. Letting $H(\seq{0,p}) = \seq{a,b}$, this means that $a \in \dom(\eval)$ and $\eval(a) \in \dom(\delta_{Y_0})$, and hence by definition,
		\[
			\delta_{\delta_{\infcomp{Y_0} \times \infcomp{Y_1}}}  H(\seq{0,p}) = \seq{\delta_{Y_0} \eval (a), z}.
		\]
		We conclude that $\delta_{Y_0} \eval (a) \in f \delta_{X_0}(p)$, but since $\delta_{Y_0} \eval(a) = \delta_{Y_0} \Psi_0 H \Phi_0(p)$, this is what we wanted.
	\end{proof}
	
	\begin{lemma}\label{lem:main2}
		Let $h : \subseteq (U,\delta_{U}) \rightrightarrows (V,\delta_{V})$ be a partial multifunction on represented spaces, and suppose $f \sured h$ and $g \sured h$. Then $f \sjoin g \sured h$.
	\end{lemma}
	
	\begin{proof}
		Suppose $f \sured h$ via $\Phi_0$ and $\Psi_0$, and $g \sured h$ via $\Phi_1$ and $\Psi_1$.	Let $\Phi : \subseteq 2^\omega \to 2^\omega$ be the map with domain all pairs $\seq{i,p}$ for $i < 2$ and $p \in \dom(\Phi_i)$, and with $\Phi(\seq{i,p}) = \Phi_i(p)$. Let $\Psi : 2^\omega \to 2^\omega$ be the map $q \mapsto \seq{\monapprox_{\Psi_0(q)}, \monapprox_{\Psi_1(q)}}$. We claim that $f \sjoin g \sured h$ via $\Phi$ and $\Psi$.
		
		Suppose $H \vdash h$ and fix any element in the domain of $(f \sjoin g) (\delta_{X_0 \sqcup X_1})$, which must have the form $\seq{i,p}$ for some $i < 2$. Without loss of generality, assume $i = 0$; a symmetric argument works if $i = 1$. We aim to show that
		\[
			\delta_{\infcomp{Y_0} \times \infcomp{Y_1}} \Psi H \Phi(\seq{0,p}) \in (f \sjoin g)(\delta_{X_0 \sqcup X_1})(\seq{0,p}).
		\]
		Since $\delta_{\infcomp{Y_0} \times \infcomp{Y_1}} \Psi H \Phi(\seq{0,p}) = \delta_{\infcomp{Y_0} \times \infcomp{Y_1}} \Psi H \Phi_0(p)$ and $(f \sjoin g)(\delta_{X_0 \sqcup X_1})(\seq{0,p}) = f\delta_{X_0}(p) \times \infcomp{Y_1}$, this is equivalent to showing
		\begin{equation}\label{eq:1}
			\delta_{\infcomp{Y_0} \times \infcomp{Y_1}} \Psi H \Phi_0(p) \in f\delta_{X_0}(p) \times \infcomp{Y_1}.
		\end{equation}
		%First, notice that $\Phi_0(p) \in \dom(h \delta_U)$.
		Now $\seq{0,p} \in \dom( (f \sjoin g) (\delta_{X_0 \sqcup X_1}))$, so $p \in \dom(f \delta_{X_0})$ by definition. And since $f \sured h$ via $\Phi_0$ and $\Psi_0$ and $H \vdash h$, this implies that
		\begin{equation}\label{eq:2}
			\delta_{Y_0} \Psi_0 H \Phi_0(p) \in f \delta_{X_0} (p).
		\end{equation}
		In particular, this means that $\Psi_0 H \Phi_0(p) \in 2^\omega$, so by Proposition~\ref{prop:approx}, $\monapprox_{\Psi_0 H \Phi_0(p)}$ is total, and by Proposition~\ref{prop:eval},
		\[
			\eval(\monapprox_{\Psi_0 H \Phi_0(p)}) = \Psi_0 H \Phi_0(p).
		\]
		It also means that $\Psi_0 H \Phi_0(p) \in \dom(\delta_{Y_0})$, and so $\eval(\monapprox_{\Psi_0 H \Phi_0(p)}) \in \dom(\delta_{Y_0})$. Now by definition, for some $z \in \infcomp{Y_1}$, we have
		\begin{eqnarray*}
			\delta_{\infcomp{Y_0} \times \infcomp{Y_1}} \Psi H \Phi_0(p) & = & \delta_{\infcomp{Y_0} \times \infcomp{Y_1}}(\seq{\monapprox_{\Psi_0 H \Phi_0(p)}, \monapprox_{\Psi_1 H \Phi_0(p)}})\\
			& = &  (\delta_{Y_0} \eval (\monapprox_{\Psi_0 H \Phi_0(p)}), z)\\
			& = &  (\delta_{Y_0} \Psi_0 H \Phi_0(p), z).
		\end{eqnarray*}
		Combining this with \eqref{eq:2} now gives \eqref{eq:1}. 
	\end{proof}
	
	The proof of the theorem is complete.
\end{proof}

\begin{corollary}\label{cor:lattice}
	The strong Weihrauch degrees form a bounded lattice under $\sured$, with $\sjoin$ as supremum and $\sqcap$ as infimum.
\end{corollary}

As noted above, the $\sqcup$ operation does not give the supremum in the strong Weihrauch degrees, so $\sjoin$ and $\sqcup$ are in general different. However, as the next proposition shows, this is no longer the case if we move from strong Weihrauch degrees to the more general setting of (non-strong) Weihrauch degrees.

\begin{proposition}\label{prop:equiv}
	Let $f : \subseteq (X_0,\delta_{X_0}) \rightrightarrows (Y_0,\delta_{Y_0})$ and $g : \subseteq (X_1,\delta_{X_1}) \rightrightarrows (Y_1,\delta_{Y_1})$ be partial multifunctions on represented spaces. Then $f \sjoin g \uequiv f \sqcup g$.
\end{proposition}

\begin{proof}
	Since $f, g \sured f \sjoin g$ by Theorem~\ref{thm:main}, and $f \sqcup g$ is the supremum of $f$ and $g$ under $\ured$, it follows that $f \sqcup g \ured f \sjoin g$. So, we only need to show that $f \sjoin g \ured f \sqcup g$, and in fact, we show that $f \sjoin g \sured f \sqcup g$. Let $\Phi : 2^\omega \to 2^\omega$ be the identity functional, and let $\Psi : 2^\omega \to 2^\omega$ be defined by
	\[
		\Psi(\seq{0,q}) = \seq{\monapprox_{\Phi(q)},0^\omega}
	\]
	and
	\[
		\Psi(\seq{1,q}) = \seq{0^\omega, \monapprox_{\Phi(q)}}
	\]
	for all $p,q \in 2^\omega$. Fix $H \vdash f \sqcup g$, and any element in the domain of $(f \sjoin g)(\delta_{X_0 \sqcup X_1})$, which must have the form $\seq{i,p}$ for some $i < 2$. We assume $i = 0$; the case $i = 1$ follows by a symmetric argument. We must show that
	\[
		\delta_{\infcomp{Y_0} \times \infcomp{Y_1}} \Psi H \Phi(\seq{0,p}) = \delta_{\infcomp{Y_0} \times \infcomp{Y_1}} \Psi H (\seq{0,p}) \in (f \sjoin g)(\delta_{X_0 \sqcup X_1})(\seq{0,p}).
	\]
	By definition,
	\[
		(f \sjoin g)(\delta_{X_0 \sqcup X_1})(\seq{0,p}) = (f \sjoin g)(\seq{0,\delta_{X_0}(p)}) = f\delta_{X_0}(p) \times \infcomp{Y_1},
	\]
	so the above is equivalent to
	\begin{equation}\label{eq:W_1}
		\delta_{\infcomp{Y_0} \times \infcomp{Y_1}} \Psi H (\seq{0,p}) \in f\delta_{X_0}(p) \times \infcomp{Y_1}.
	\end{equation}
	Since $H \vdash f \sqcup g$, we have
	\[
		\delta_{Y_0 \sqcup Y_1} H(\seq{0,p}) \in (f \sqcup g)(\delta_{X_0 \sqcup X_1})(\seq{0,p}) = \{0\} \times f\delta_{X_0}(p).
	\]
	Thus, it must be that $\delta_{Y_0 \sqcup Y_1} H(\seq{0,p}) = \seq{0,y}$ for some $y \in f\delta_{X_0}(p)$, and hence that $H(\seq{0,p}) = \seq{0,q}$ for some $q$ with $\delta_{Y_0}(q) = y$. Thus, we have
	\begin{equation}\label{eq:W_2}
		\delta_{\infcomp{Y_0} \times \infcomp{Y_1}} \Psi H (\seq{0,p}) = \delta_{\infcomp{Y_0} \times \infcomp{Y_1}} \Psi (\seq{0,q}) = \delta_{\infcomp{Y_0} \times \infcomp{Y_1}} (\monapprox_{\Phi(q)}, 0^\omega).
	\end{equation}
	Since $\Phi(q) = q$ we have $\eval(\monapprox_{\Phi(q)}) = \Phi(q) = q$, so $\eval(\monapprox_{\Phi(q)}) \in \dom(\delta_0)$. We conclude that
	\[
		\delta_{\infcomp{Y_0} \times \infcomp{Y_1}} (\monapprox_{\Phi(q)}, 0^\omega) = \seq{\delta_{Y_0}(q),z}
	\]
	for some $z \in \infcomp{Y_1}$. Combining this with \eqref{eq:W_2} gives \eqref{eq:W_1}.
\end{proof}

\section{Distributivity}\label{S:distr}

Our aim is to examine some of the lattice-theoretic properties of the strong Weihrauch degrees. Recall that a lattice $\mathcal{L} =(L,\vee,\wedge)$ is \emph{distributive} if the operations of join and meet distribute over one another, i.e., if for all $a,b,c \in L$ we have $(a \vee b) \wedge c = (a \wedge c) \vee (b \wedge c)$.
%$(a \wedge b) \vee c = (a \wedge c) \vee (b \wedge c)$.
As noted above, the Weihrauch lattice is distributive. By contrast, we will show below that the strong Weihrauch lattice is not.

We begin with the following result, showing that one half of the distributivity identity in the strong Weihrauch degrees does indeed hold under $\sured$, while the other holds if we replace $\sjoin$ by $\sqcup$.

\begin{proposition}\label{prop:boring}
	Let $f : \subseteq (X_0,\delta_{X_0}) \rightrightarrows (Y_0,\delta_{Y_0})$, $g : \subseteq (X_1,\delta_{X_1}) \rightrightarrows (Y_1,\delta_{Y_1})$, and $h : (U,\delta_U) \to (V,\delta_V)$ be partial multifunctions on represented spaces. Then we have:
	\begin{enumerate}
		\item $(f \sqcap h) \sjoin (g \sqcap h) \sured (f \sjoin g) \sqcap h$;
		\item $(f \sqcup g) \sqcap h \sured (f \sqcap h) \sqcup (g \sqcap h)$.
	\end{enumerate}
\end{proposition}

\begin{proof}
	For part~(1), let $\Phi : \subseteq 2^\omega \to 2^\omega$ be the map $\seq{i,\seq{p_0,p_1}} \mapsto \seq{\seq{i,p_0},p_1}$ for all $i < 2$ and all $p_0,p_1 \in 2^\omega$. Let $\operatorname{id} : 2^\omega \to 2^\omega$ be the identity functional, and let $\Psi : \subseteq 2^\omega \to 2^\omega$ be the map given by
	\[	
		\Psi(\seq{0,\seq{q_0,q_1}}) = \seq{\monapprox_{\operatorname{id}(\seq{0,\eval(q_0)})}, \monapprox_{\operatorname{id}(\seq{0,\eval(q_1)})}}
	\]
	for all $q_0,q_1 \in 2^\omega$, and
	\[
		h(\seq{1,q}) = \seq{\monapprox_{\operatorname{id}(\seq{1,q})}, \monapprox_{\operatorname{id}(\seq{1,q})}}
	\]
	for all $q \in 2^\omega$. We claim that $(f \sqcap h) \sjoin (g \sqcap h) \sured (f \sjoin g) \sqcap h$ via $\Phi$ and $\Psi$.
	
	Fix any $H \vdash (f \sjoin g) \sqcap h$, and any element in the domain of $((f \sqcap h) \sjoin (g \sqcap h)) \delta_{(X_0 \times U) \sqcup (X_1 \times U)}$. This must have the form $\seq{i,\seq{p_0,p_1}}$ for some $i < 2$ and some $p_0$ in the domain of $f \delta_{X_0}$ if $i = 0$ or $g \delta_{X_1}$ if $i = 1$, and some $p_1$ in the domain of $h \delta_{U}$. Assume $i = 0$; a symmetric argument works if $i = 1$. We must then show that
	\[
		\delta_{\infcomp{Y_0 \sqcup V} \times \infcomp{Y_1 \sqcup V}} \Psi H \Phi (\seq{0,\seq{p_0,p_1}}) \in (f \sqcap h) \sjoin (g \sqcap h) \delta_{(X_0 \times U) \sqcup (X_1 \times U)} (\seq{0,\seq{p_0,p_1}}).
	\]
	We have
	\[
		\delta_{\infcomp{Y_0 \sqcup V} \times \infcomp{Y_1 \sqcup V}} \Psi H \Phi (\seq{0,\seq{p_0,p_1}}) = \delta_{\infcomp{Y_0 \sqcup V} \times \infcomp{Y_1 \sqcup V}} \Psi H (\seq{\seq{0,p_0},p_1}),
	\]
	and
	\begin{eqnarray*}
		(f \sqcap h) \sjoin (g \sqcap h) \delta_{(X_0 \times U) \sqcup (X_1 \times U)} (\seq{0,\seq{p_0,p_1}})\\
		= (f \sqcap h) \sjoin (g \sqcap h) (\seq{0,\seq{\delta_{X_0}(p_0),\delta_U(p_1)}})\\
		= (f \sqcap h)(\seq{\delta_{X_0}(p_0),\delta_U(p_1)}) \times \infcomp{Y_1 \sqcup V}\\
		= (f \delta_{X_0}(p_0) \sqcup h \delta_U(p_1)) \times \infcomp{Y_1 \sqcup V}.
	\end{eqnarray*}
	Thus, it is enough to show that
	\[
		\delta_{\infcomp{Y_0 \sqcup V} \times \infcomp{Y_1 \sqcup V}} \Psi H (\seq{\seq{0,p_0},p_1}) \in (f \delta_{X_0}(p_0) \sqcup h \delta_U(p_1)) \times \infcomp{Y_1 \sqcup V}.
	\]
	
	Now since $H \vdash (f \sjoin g) \sqcap h$, we have
	\begin{eqnarray*}
		\delta_{(\infcomp{Y_0} \times \infcomp{Y_1}) \sqcup V} H(\seq{\seq{0,p_0},p_1}) & \in & ((f \sjoin g) \sqcap h) \delta_{(X_0 \sqcup X_1) \times U} (\seq{\seq{0,p_0},p_1})\\
		& = & (f \delta_{X_0} (p_0) \times \infcomp{Y_1}) \sqcup h \delta_U (p_1).
	\end{eqnarray*}
	Therefore, $\delta_{(\infcomp{Y_0} \times \infcomp{Y_1}) \sqcup V} H(\seq{\seq{0,p_0},p_1})$ is either $\seq{0,\seq{y,z}}$ for some $y \in f \delta_{X_0} (p_0)$ and $z \in \infcomp{Y_1}$, or $\seq{1,v}$ for some $v \in h \delta_U (p_1)$. In the first case, since $y \neq \diff_{Y_0}$, it must be that $H(\seq{\seq{0,p_0},p_1}) = \seq{0,\seq{a,q}}$ for some $a \in \dom(\eval)$ with $\eval(a) \in \dom(\delta_{Y_0})$, meaning $\delta_{Y_0}(\eval(a)) = y$. Consequently, $\monapprox_{\operatorname{id}(\seq{0,\eval(a)})} \in \dom(\eval)$ and $\eval(\monapprox_{\operatorname{id}(\seq{0,\eval(a)})}) = \seq{0,\eval(a)} \in \dom(\delta_{Y_0 \sqcup V})$.
	It follows that
	\begin{eqnarray*}
		\delta_{\infcomp{Y_0 \sqcup V} \times \infcomp{Y_1 \sqcup V}} \Psi H (\seq{\seq{0,p_0},p_1}) & = & \delta_{\infcomp{Y_0 \sqcup V} \times \infcomp{Y_1 \sqcup V}} \Psi (\seq{0,\seq{a,q}})\\
		& = & \delta_{\infcomp{Y_0 \sqcup V} \times \infcomp{Y_1 \sqcup V}} (\seq{\monapprox_{\operatorname{id}(\seq{0,\eval(a)})}, \monapprox_{\operatorname{id}(\seq{0,\eval(q)})}})\\
		& = & \seq{\delta_{\infcomp{Y_0 \sqcup V}}(\monapprox_{\operatorname{id}(\seq{0,\eval(a)})}), \delta_{\infcomp{Y_1 \sqcup V}}(\monapprox_{\operatorname{id}(\seq{0,\eval(q)})})}\\
		& = & \seq{\delta_{{Y_0 \sqcup V}}(\eval(\monapprox_{\operatorname{id}(\seq{0,\eval(a)})})), \delta_{\infcomp{Y_1 \sqcup V}}(\monapprox_{\operatorname{id}(\seq{0,\eval(q)})})}\\
		& = & \seq{\seq{0,\delta_{Y_0}(\eval(a))}, \delta_{\infcomp{Y_1 \sqcup V}}(\monapprox_{\operatorname{id}(\seq{0,\eval(q)})})}\\
		& \in & (f \delta_{X_0}(p_0) \sqcup h \delta_U(p_1)) \times \infcomp{Y_1 \sqcup V},
	\end{eqnarray*}
	which is what was to be shown. In the second case, if $\delta_{(\infcomp{Y_0} \times \infcomp{Y_1}) \sqcup V} H(\seq{\seq{0,p_0},p_1}) = \seq{1,v}$ for some $v \in h \delta_U (p_1)$, it must be that $H(\seq{\seq{0,p_0},p_1}) = \seq{1,q}$ for some $q$ with $\delta_V(q) = v$. We then have
	\begin{eqnarray*}
		\delta_{\infcomp{Y_0 \sqcup V} \times \infcomp{Y_1 \sqcup V}} \Psi H (\seq{\seq{0,p_0},p_1}) & = & \delta_{\infcomp{Y_0 \sqcup V} \times \infcomp{Y_1 \sqcup V}} \Psi (\seq{1,q})\\
		& = & \delta_{\infcomp{Y_0 \sqcup V} \times \infcomp{Y_1 \sqcup V}} (\seq{\monapprox_{\operatorname{id}(\seq{1,q})}, \monapprox_{\operatorname{id}(\seq{1,q})}})\\
		& = & \seq{ \delta_{\infcomp{Y_0 \sqcup V}}(\monapprox_{\operatorname{id}(\seq{1,q})}), \delta_{\infcomp{Y_1 \sqcup V}}(\monapprox_{\operatorname{id}(\seq{1,q})}) }\\
		& = & \seq{ \delta_{{Y_0 \sqcup V}}(\eval (\monapprox_{\operatorname{id}(\seq{1,q})})), \delta_{{Y_1 \sqcup V}}(\eval (\monapprox_{\operatorname{id}(\seq{1,q})})) }\\
		& = & \seq{ \delta_{{Y_0 \sqcup V}}(\seq{1,q})), \delta_{{Y_1 \sqcup V}}(\seq{1,q}) }\\
		& = & \seq{\seq{1,\delta_V(q)}, \seq{1,\delta_V(q)}}\\
		& \in & (f \delta_{X_0}(p_0) \sqcup h \delta_U(p_1)) \times \infcomp{Y_1 \sqcup V}.
	\end{eqnarray*}
	This completes the proof of part~(1).
	
	Part~(2) can be proved similarly, but can also be observed more directly: the standard proof that $(f \sqcup g) \sqcap h \ured (f \sqcap h) \sqcup (g \sqcap h)$ actually shows $(f \sqcup g) \sqcap h \sured (f \sqcap h) \sqcup (g \sqcap h)$. We omit the details.
\end{proof}

In the remainder of this section, we will be dealing with multifunctions on Cantor space. We regard Cantor space as a represented space under the trivial (identity) representation, which we also denote by $\delta_{2^\omega}$ for consistency of notation when viewing it as a representation. In this setting, we can then use the following alternative definition of Weihrauch reducibility, which will be slightly easier to work with.

\begin{definition}
	Let $f: \subseteq 2^\omega \rightrightarrows 2^\omega$ and $g: \subseteq 2^\omega \rightrightarrows 2^\omega$ be multifunctions.
	\begin{itemize}
		\item $f \ured g$ if there are Turing functionals $\Phi, \Psi : \subseteq 2^\omega \to 2^\omega$ such that for every $p \in \dom(f)$, $\Phi(p) \in \dom(g)$ and $\Psi(\seq{p,q}) \in f(p)$ for every $q \in g(\Phi(p))$.
		\item $f \sured g$ if there are Turing functionals $\Phi, \Psi : \subseteq 2^\omega \to 2^\omega$ such that for every $p \in \dom(f)$, $\Phi(p) \in \dom(g)$ and $\Psi(q) \in f(p)$ for every $q \in g(\Phi(p))$.
	\end{itemize}
\end{definition}

\noindent See~\cite{DDHMS-2016}, Appendix A, for a discussion and comparison of this approach to that in Definition~\ref{def:weihr}, and for a proof of the equivalence of the two.

The following observation will be useful.

\begin{lemma}\label{lem:simplejoin}
	Let $f: \subseteq 2^\omega \rightrightarrows 2^\omega$ and $g: \subseteq 2^\omega \rightrightarrows 2^\omega$ be given. Let $h : \subseteq 2^\omega \rightrightarrows 2^\omega$ be the multifunction with domain consisting of all pairs $\seq{0,p}$ for $p \in \dom(f)$, and $\seq{1,p}$ for $p \in \dom(g)$, and satisfying the following:
	\begin{itemize}
		\item $h(\seq{0,p})$ consists of all pairs $\seq{a,q}$ with $a,q \in 2^\omega$, where $a$ is a monotone approximation with $a \in \dom(\eval)$ and $\eval(a) \in f(p)$;
		\item $h(\seq{1,p})$ consists of all pairs $\seq{q,a}$ with $q,a \in 2^\omega$, where $a$ is a monotone approximation with $a \in \dom(\eval)$ and $\eval(a) \in g(p)$.
	\end{itemize}
	Then $h \suequiv f \sjoin g$.
\end{lemma}

\begin{proof}
	It is clear that $f,g \sured h$, hence $f \sjoin g \sured h$. We thus only need to show that $h \sured f \sjoin g$, and we claim that this is so via the identity map, $\operatorname{id} : 2^\omega \to 2^\omega$, in both directions. To see this, fix $H \vdash f \sjoin g$, and any element in the domain of $h \delta_{2^\omega} = h$. Without loss of generality, assume this has the form $\seq{0,p}$ for some $p \in \dom(f)$; a symmetric argument works if it has the form $\seq{1,p}$ for $p \in \dom(g)$. We must show that
	\[
		\operatorname{id} H \operatorname{id} (\seq{0,p}) = H (\seq{0,p}) \in h \delta_{2^\omega} (p) = h(p). 
	\]
	Since $H \vdash f \sjoin g$, we have that
	\[
		\delta_{ \infcomp{2^\omega} \times \infcomp{2^\omega} } H(\seq{0,p}) \in (f \sjoin g)(\delta_{2^\omega \sqcup 2^\omega})(\seq{0,p}) = (f \sjoin g)(\seq{0,p}) = f(p) \times \infcomp{2^\omega}.
	\]
	Thus, $\delta_{ \infcomp{2^\omega} \times \infcomp{2^\omega} }H(\seq{0,p}) = \seq{y,z}$ for some $y \in f(p)$ and $z \in \infcomp{2^\omega}$. In particular, $y \neq \diff_{2^\omega}$, so by definition, $H(\seq{0,p}) = \seq{a,q}$ for some $a \in \dom(\eval)$ with $y = \eval(a)$. In other words, $\eval(a) \in f(p)$, whence we conclude that $\seq{a,q} \in h(\seq{0,p})$, as desired.
\end{proof}

We now come to our main result in this section, in which we will demonstrate that the inequality of Proposition \ref{prop:boring} cannot in general be reversed.

\begin{theorem}\label{thm:dist}
	The lattice of strong Weihrauch degrees is not distributive.
\end{theorem}

\begin{proof}
	Choose $p_0,p_1,p_2,q_0,q_1,q_2 \in 2^\omega$ with the following properties:
	\begin{itemize}
		\item $p_0 \nTred \seq{p_1,p_2}$ and $p_1 \nTred \seq{p_0, p_2}$;
		\item $q_0 \nTred q_2$ and $q_2 \nTred q_1$.
	\end{itemize}
	Define $f,g,h : \subseteq 2^\omega \to 2^\omega$ to be $\{p_0\} \mapsto \{q_0\}$, $\{p_1\} \mapsto \{q_1\}$, and $\{p_2\} \mapsto \{q_2\}$, respectively. We claim that
	\[
		(f \sjoin g) \sqcap h \nsured (f \sqcap h) \sjoin (g \sqcap h),
	\]
	which gives the theorem.
	
	Seeking a contradiction, suppose $\Phi$ and $\Psi$ actually witness the reduction above. For each $i < 2$, we must then have that
	\[
		\Phi(\seq{\seq{i,p_i},p_2}) = \seq{i,\seq{p_i,p_2}}.
	\]
	For otherwise there would be $i,j < 2$ with $i \neq j$ and
	\[
		\Phi(\seq{\seq{i,p_i},p_2}) = \seq{j,\seq{p_j,p_2}},
	\]
	whence we would have $p_j \Tred \seq{p_i,p_2}$, contrary to our assumption.
	
	Now fix any monotone approximation $a_2$ such that $a_2 \Tred q_2$ and $a_2 \in \dom(\eval)$ and $\eval(a_2) = \seq{0,q_2}$. By Lemma~\ref{lem:simplejoin}, $\seq{a_2,0^\omega}$ is a solution to $\seq{0,\seq{p_0,p_2}}$ in $(f \sqcap h) \sjoin (g \sqcap h)$. Hence, $\Psi(\seq{a_2,0^\omega})$ must be a solution to $\seq{\seq{0,p_0},p_2}$ in $(f \sjoin g) \sqcap h$, and hence be equal either to $\seq{1,q_2}$, or else, again by the lemma, to $\seq{0,\seq{a_0,c}}$ for some monotone approximation $a_0$ with $a_0 \in \dom(\eval)$ and $\eval(a_0) = q_0$. In the latter case, we would have $q_0 \Tred a_0 \Tred a_2 \Tred q_2$, contradicting our choice of $q_0$ and $q_2$. So it must be the first case that applies.
	
	Let $u$ be the use of computing that the first coordinate of $\Psi(\seq{a_2,0^\omega})$ is $1$. Fix a monotone approximation $a_1$ with $\min a_1 > u$ and $a_1 \Tred q_1$ and $a_1 \in \dom(\eval)$ and $\eval(a_1) = q_1$. Let $q = (a \res u) 0^\omega$, noting that $q$ is computable. Then $\seq{q,a_1}$ is a solution to $\seq{1,\seq{p_1,p_2}}$ in $(f \sqcap h) \sjoin (g \sqcap h)$, and $\Psi(\seq{q,a_1})$ must therefore be a solution to $\seq{\seq{1,p_1},p_2}$ in $(f \sjoin g) \sqcap h$. But the first coordinate of $\Psi(\seq{q,a_1})$ agrees with $\Psi(\seq{a_2,0^\omega})$, so we must have that $\Psi(\seq{q,a_1}) = \seq{1,q_2}$. We then have that $q_2 \Tred \seq{q,a_1} \Tred a_1 \Tred q_1$, contradicting our choice of $q_1$ and $q_2$.
\end{proof}

\begin{corollary}
	The strong Weihrauch lattice is not isomorphic to the Weihrauch lattice.	
\end{corollary}

\begin{remark}\label{rem:nonunif}
	The above proof does not go through if $\sured$ is replaced by the related \emph{strong computable reducibility} ($\scred$). Along with \emph{computable reducibility} ($\cred$), these form non-uniform variants of strong Weihrauch and Weihrauch reducibility, respectively. (See, e.g.,~\cite{Dzhafarov-2016}, Definition 1.1, for the precise definitions.) The corresponding algebraic structures have not previously been studied, but it is easy to see that they form lattices under $\sqcup$ and $\sqcap$, just as in the Weihrauch case. The distributivity of $\cred$ then follows from the distributivity of $\ured$. For $\scred$, it follows from Proposition \ref{prop:boring}, together with Proposition \ref{prop:equiv}, the proof of which also shows that $\sjoin$ and $\sqcup$ are the same up to strong computable equivalence. We can conclude that the non-distributivity of the strong Weihrauch lattice is not a feature of uniformity alone, or of denying access to the original instance alone, but rather of the two properties in combination.
\end{remark}

We finish by showing that, in spite of Theorem \ref{thm:dist}, the strong Weihrauch lattice is nonetheless very rich. Recall that a set $A \subseteq 2^\omega$ is \emph{Medvedev reducible} to $B \subseteq 2^\omega$ if there is a functional $\Phi$ such that $\Phi(p) \in A$ for every $p \in B$. The \emph{Medvedev degrees} are the equivalence classes under this reducibility. It is easy to see that these form a lattice, with $A \times B$ serving as the join of $A$ and $B$, and $A \sqcup B$ serving as the meet. Sorbi~\cite[Lemma 6.1]{Sorbi-1990} has shown that every countable distributive lattice embeds into the Medvedev lattice. It is easy to see that the Medvedev degrees embed into the Weihrauch degrees as a partial order, via the embedding sending $A \subseteq 2^\omega$ to $0^\omega \mapsto A$, but it was shown by Higuchi and Pauly~\cite[Corollary 5.3]{HP-2013} that this is not a lattice embedding. However, they also established the following reverse-embedding result.

\begin{proposition}[Higuchi and Pauly~\cite{HP-2013}, Lemma 5.6]
	The Medvedev lattice reverse-embeds into the Weihrauch lattice.
\end{proposition}

\noindent The proof uses the following embedding, originally due to Brattka (see~\cite{HP-2013}, Definition 5.5).

\begin{definition}
	Given $A \subseteq 2^\omega$, let $d_A : \subseteq 2^\omega \to 2^\omega$ be the map $A \to \{0^\omega\}$.
\end{definition}

We show that the same map works to reverse-embed the Medvedev degrees into the strong Weihrauch degrees as lattices.

\begin{proposition}
	The Medvedev lattice reverse-embeds into the strong Weihrauch lattice.
\end{proposition}

\begin{proof}
	Given $A,B \subseteq 2^\omega$, if $\Phi$ is a Turing functional such that $\Phi(p) \in A$ for every $p \in B$ then $d_B \sured d_A$ via $\Phi$ and the identity. Conversely, if $d_B \sured d_A$ via $\Phi$ and $\Psi$ then $\Phi(p) \in A$ for every $p \in B$. We have $d_A,d_B \sured d_{A \sqcup B}$, so $d_A \sjoin d_B \sured d_{A \sqcup B}$. In the other direction, we have $d_{A \sqcup B} \sured d_A \sjoin d_B$ via the identity and the constant $q \mapsto 0^\omega$ map. Similarly, we have $d_{A \times B} \sured d_A,d_B$ and hence $d_{A \times B} \sured d_A \sqcap d_B$. And $d_A \sqcap d_B \sured d_{A \times B}$ via the identity map and the map $q \mapsto \seq{0,0^\omega}$.
\end{proof}

\begin{corollary}
	Every countable distributive lattice can be embedded into the strong Weihrauch lattice.
\end{corollary}

\end{document}